\documentclass[11pt]{article}      

\usepackage{amssymb}      
\usepackage{amsmath}      
\usepackage{amsthm}      
\usepackage{amsbsy}      
\usepackage{amscd}      
\usepackage[dvips]{graphicx}

\theoremstyle{plain}      
\newtheorem{theorem}{Theorem}[section]      
\newtheorem{lemma}{Lemma}[section]      
\newtheorem{corollary}[theorem]{Corollary}      
\newtheorem{proposition}{Proposition}[section]

\newtheorem{definition}{Definition}[section]          
\theoremstyle{remark}      
\newtheorem{remark}{Remark}[section]

\newcommand{\Z}{{\mathbb{Z}}}   
   
\newcommand{\C}{{\mathbb{C}}}

\begin{document}

\date{\today}

\title{Free subgroups within the images of quantum representations}         
\author{      
\begin{tabular}{cc}      
 Louis Funar &  Toshitake Kohno\\      
\small \em Institut Fourier BP 74, UMR 5582       
&\small \em IPMU, Graduate School of Mathematical Sciences\\      
\small \em University of Grenoble I &\small \em The University of Tokyo    \\      
\small \em 38402 Saint-Martin-d'H\`eres cedex, France      
&\small \em 3-8-1 Komaba, Meguro-Ku, Tokyo 153-8914 Japan \\      
\small \em e-mail: {\tt funar@fourier.ujf-grenoble.fr}      
& \small \em e-mail: {\tt kohno@ms.u-tokyo.ac.jp} \\      
\end{tabular}      
}

\maketitle 

\begin{abstract}
We prove that, except for a few explicit 
roots of unity, the quantum image of any Johnson subgroup of the 
mapping class group contains an  explicit free non-abelian subgroup.

\vspace{0.1cm}
 

\end{abstract}

\section{Introduction and statements}

\vspace{0.2cm}\noindent 
The aim of this paper is 
to study  the images of the mapping class groups   
by quantum representations.
Some results in this direction are already known. 
We refer the reader to \cite{RT} and \cite{Kohno}
for earlier treatments of quantum representations.
In \cite{F} we proved that the images are infinite and non-abelian (for all but finitely 
many explicit cases) using earlier results of Jones who proved in 
\cite{Jones} that the same holds true  
for the braid group representations factorizing through the 
Temperley-Lieb algebra at roots of unity. Masbaum then found in \cite{Mas2} 
explicit elements of infinite order in the image.   
General arguments concerning Lie groups actually show 
that the image should contain a free 
non-abelian group. Furthermore, Larsen and Wang showed (see \cite{LW}) 
that the image  of the quantum representations of the mapping class groups at 
roots of unity of the form 
$\exp\left(\frac{2\pi i}{4r}\right)$, for prime $r\geq 5$,  
is dense in the projective unitary group. 

\vspace{0.2cm}\noindent 
In order to be precise we have to specify the quantum 
representations we are considering. Recall that in \cite{BHMV} 
the authors defined the TQFT functor $\mathcal V_{p}$, for every $p\geq 3$  
and a primitive root of unity $A$ of order $2p$.
These TQFT should correspond to the so-called 
$SU(2)$-TQFT, for even $p$ and to 
the $SO(3)$-TQFT, for odd $p$ (see also \cite{LW} for another 
$SO(3)$-TQFT). 

\begin{definition}\label{qrep}
Let $p\in\Z_+$, $p\geq 3$, such that $p\not\equiv 2({\rm mod}\: 4)$. 
The  quantum representation $\rho_p$ 
is the projective representation of  the mapping class group 
associated to the TQFT $\mathcal V_{\frac{p}{2}}$ for even $p$
and $\mathcal V_{p}$ for odd $p$,   
corresponding to the following choices of the root of unity: 
\[ A_p=\left\{\begin{array}{ll}
-\exp\left(\frac{2\pi i}{p}\right), & {\rm if}\: p\equiv 0({\rm mod}\: 4);\\
-\exp\left(\frac{(p+1)\pi i}{p}\right) , & {\rm if}\: p\equiv 1({\rm mod}\: 2).\\
\end{array}\right. \]
Notice that $A_p$ is a primitive root of unity of order $p$ when $p$ is even 
and of order $2p$ otherwise.  
\end{definition}

\begin{remark}\label{order}
The eigenvalues of a Dehn twist in the TQFT $\mathcal V_p$ i.e.,   
the entries of the diagonal $T$-matrix  
are of the form $\mu_l=(-A_p)^{l(l+2)}$, where $l$ belongs to the 
set of admissible colors (see \cite{BHMV}, 4.11). 
The set of admissible colors 
is $\{0,1,2,\ldots,\frac{p}{2}-2\}$, for even $p$ and is 
$\{0,2,4,\ldots, p-3\}$ for odd $p$.  
Therefore the order of the image of a Dehn twist by $\rho_p$ 
is $p$. 
\end{remark}

\vspace{0.2cm}\noindent 
We will now consider the Johnson filtration by the 
subgroups $I_g(k)$ of the mapping class group $M_g$ of the closed orientable 
surface of genus $g$, consisting of those 
elements having a trivial outer action on the 
$k$-th nilpotent quotient of the fundamental 
group of the surface, for  some $k\in\Z_+$. 
As is well-known the Johnson filtration 
shows up within the framework of finite type invariants of 3-manifolds 
(see e.g. \cite{GL}). 

\vspace{0.2cm}\noindent 
Our next result shows that the image is large 
in the  following sense (see also Propositions \ref{Johnsonfree} and 
\ref{excases}):

\begin{theorem}\label{Johnson}
Assume that $g\geq 3$ and $p\not\in\{3, 4, 8, 12, 16, 24\}$ 
or $g=2$, $p$ is even and $p\not\in\{4, 8, 12, 16, 24, 40\}$. 
Then for any $k$, the image $\rho_p(I_g(k))$ of the $k$-th 
Johnson subgroup by the quantum 
representation $\rho_p$ contains a free non-abelian group.   
\end{theorem}

\vspace{0.2cm}\noindent 
The idea of proof  for this theorem 
is to embed a pure braid group within the mapping class 
group and to show that its image is large. Namely,  
a 4-holed sphere suitably embedded in the surface leads to 
an embedding of the pure braid group $PB_3$ in the mapping class 
group.  The quantum representation contains a particular  
sub-representation which is the restriction of 
Burau's representation (see \cite{F}) to a free subgroup of 
$PB_3$. 
One way to obtain elements of the Johnson filtration is to consider 
elements of the lower central series of $PB_3$ and extend them 
to all of the surface by identity. Therefore it suffices to find  
free non-abelian subgroups  in the image of the 
lower central series of $PB_3$ by Burau's representation at 
roots of unity in order to prove Theorem \ref{Johnson}. 

\vspace{0.2cm}\noindent
The analysis of the contribution of mapping classes supported on  
small sub-surfaces of a surface, which are usually holed spheres,  
to various subgroups of the mapping class groups 
was also used in an unpublished paper by T.~Oda 
and J.~Levine (see \cite{Levine}) for obtaining lower bounds for 
the ranks of the graded quotients of the Johnson filtration. 

\vspace{0.2cm}\noindent
Our construction also provides explicit free non-abelian subgroups 
(see Theorems \ref{free} and \ref{free2} for precise statements).

\vspace{0.2cm}\noindent 
{\bf Acknowledgements.}  We are grateful to 
J{\o}rgen Andersen,  Greg Kuperberg, Greg McShane and Gregor Masbaum   
for useful discussions and to Ian Agol for pointing out a gap in the  
previous version of this paper.
The second author is partially supported by Grant-in-Aid for Scientific
Research 23340014, Japan Society for Promotion of Science, and by World 
Premier International Research Center Initiative, MEXT, Japan.  
A part of this work was accomplished while the second
author was staying at Institut Fourier in Grenoble. He would like to thank
Institut Fourier for hospitality.

\section{Burau's representations of  $B_3$  and triangle groups}\label{tri}
\vspace{0.2cm}\noindent
Let $B_n$ denote the braid group on $n$ strands with 
the standard generators $g_1,g_2,\ldots,g_{n-1}$. Squier was interested 
to compare the kernel of Burau's representation $\beta_{q}$ at a 
$k$-th root of unity $q$ with the normal subgroup $B_n[k]$ 
of $B_n$ generated by $g_j^{k}$, $1\leq j\leq n-1$. Recall that: 

\begin{definition}
The (reduced) Burau representation $\beta:B_n\to GL(n-1,\Z[q,q^{-1}])$ 
is defined on the standard generators 
\[ \beta_q(g_1)=\left(\begin{array}{cc}
-q & 1 \\
0  & 1 \\
\end{array}
\right) \oplus {\mathbf 1}_{n-3},\]
\[ \beta_q(g_j)={\mathbf 1}_{j-2}\oplus 
\left(\begin{array}{ccc}
1 & 0 & 0 \\
q & -q & 1 \\
0 & 0  & 1 \\
\end{array}
\right) \oplus {\mathbf 1}_{n-j-2}, \:\: {\rm for} \:\: 2\leq j\leq n-2,\]
\[ \beta_q(g_{n-1})={\mathbf 1}_{n-3}\oplus 
\left(\begin{array}{cc}
1 & 0 \\
q & -q \\
\end{array}
\right). 
\]
\end{definition}

\vspace{0.2cm}\noindent
The paper \cite{FK1} is devoted to the complete 
description of the image of Burau's representation of $B_3$ at roots of unity.
Similar results were obtained in \cite{Ku,Mas2,McM}. For the sake 
of completeness we review here the essential tools from 
\cite{FK1} to be used later.  

\vspace{0.2cm}\noindent
Let us denote by $A=\beta_{-q}(g_1^2)$ and $B=\beta_{-q}(g_2^2)$ and 
$C=\beta_{-q}((g_1g_2)^3)$. 
As is well-known $PB_3$ is isomorphic to the direct product 
${\mathbb F}_2\times \Z$, where ${\mathbb F}_2$ is freely generated by 
$g_1^2$ and $g_2^2$ and the factor $\Z$ is the center of $B_3$ 
generated by $(g_1g_2)^3$. 

\vspace{0.2cm}\noindent
It is simple to check that: 
\[ A=\left(\begin{array}{cc}
q^2 & 1+q \\
0  & 1 \\
\end{array}
\right), \,\, 
B=  
\left(\begin{array}{cc}
1 & 0 \\
-q-q^2 & q^2 \\
\end{array}
\right), \,\,
C=  
\left(\begin{array}{cc}
-q^3 & 0 \\
0 & -q^3 \\
\end{array}
\right).
\]

\vspace{0.2cm}\noindent
Recall that $PSL(2,\Z)$ is the quotient of $B_3$ by its center. 
Since $C$ is a scalar matrix  
the homomorphism $\beta_{-q}:B_3\to GL(2,\C)$ 
factors to a homomorphism  
$PSL(2,\Z)\to PGL(2,\C)$.

\vspace{0.2cm}\noindent
We will be concerned below with  the subgroup 
$\Gamma_{-q}$ of $PGL(2,\C)$ generated by  the images of 
$A$ and $B$ in  $PGL(2,\C)$.
When  $\beta_{-q}$ is unitarizable,  
the group $\Gamma_{-q}$ can be viewed as a subgroup of 
the complex-unitary  group $PU(1,1)$.

\vspace{0.2cm}\noindent
Before we proceed we make a short digression on triangle groups.
Let $\Delta$ be a geodesic triangle in the hyperbolic plane of angles 
$\frac{\pi}{m},\frac{\pi}{n},\frac{\pi}{p}$, so that  
$\frac{1}{m}+\frac{1}{n}+\frac{1}{p}<1$.  
The extended triangle group $\Delta^*(m,n,p)$ is the group of isometries 
of the hyperbolic plane generated by the three reflections 
$R_1,R_2,R_3$ with respect to the edges  of $\Delta$. It is well-known that 
a presentation of  $\Delta^*(m,n,p)$ is given by 
\[  \Delta^*(m,n,p)=\langle R_1,R_2,R_3\; ; \; 
R_1^2=R_2^2=R_3^2=1,\; (R_1R_2)^m=(R_2R_3)^n=(R_3R_1)^p=1\rangle. \]
The second type of relations have a simple geometric meaning. 
In fact, the product of the reflections with respect to two 
adjacent edges is a 
rotation by the angle which is twice the angle between those edges.
The subgroup $\Delta(m,n,p)$ 
generated by the rotations $a=R_1R_2$, $b=R_2R_3$, $c=R_3R_1$ 
is a normal subgroup of index 2, which coincides  
with the subgroup of isometries preserving the orientation. 
One calls $\Delta(m,n,p)$ the triangle (also called triangular, or 
von Dyck) group 
associated to $\Delta$.
Moreover, the triangle group has the 
presentation: 
\[ \Delta(m,n,p)=\langle a,b,c\; ; \; 
a^m=b^n=c^p=1, abc=1\rangle. \]
Observe that $\Delta(m,n,p)$ also makes sense  when $m,n$ or $p$ are 
negative integers, by interpreting the associated generators as 
clockwise rotations. 
The triangle $\Delta$ is a fundamental domain for the action 
of $\Delta^*(m,n,p)$ on the hyperbolic plane. Thus 
a fundamental domain for $\Delta(m,n,p)$ consists of the 
union $\Delta^*$ of $\Delta$ with the reflection of $\Delta$ in one of its 
edges.

\begin{proposition}[\cite{FK1}]\label{even}
Let $m<k$ be such that ${\rm gcd}(m,k)=1$ where $k\geq 4$. 
Then the group 
$\Gamma_{-\exp\left(\frac{\pm 2m\pi i}{2k}\right)}$ is a triangle 
group with  the presentation:  
\[ 
\Gamma_{-\exp\left(\frac{\pm 2m\pi i}{2k}\right)}=\langle A,B; A^k=B^k=(AB)^k=1\rangle. \]
\end{proposition}

\vspace{0.2cm}\noindent
If $n$ is odd $n=2k+1$, then the group $\Gamma_{-q}$ is a quotient of the 
triangle group associated to $\Delta$, which embeds into the group 
associated to some sub-triangle $\Delta'$ of $\Delta$. 

\begin{proposition}[\cite{FK1}]\label{odd}
Let $0<m<2k+1$ be such that ${\rm gcd}(m,2k+1)=1$ and $k\geq 3$. 
Then the group 
$\Gamma_{-\exp\left(\frac{\pm 2m\pi i}{2k+1}\right)}$ is  isomorphic 
to the  triangle group  $\Delta(2,3,2k+1)$ and has the 
following presentation (in terms of our generators $A,B$): 
\[ 
\Gamma_{-\exp\left(\frac{\pm 2m\pi i}{2k+1}\right)}=
\langle A,B; A^{2k+1}=B^{2k+1}=(AB)^{2k+1}=1, \,
(A^{-1}B^k)^2=1, \, (B^kA^{k-1})^3=1 \rangle. \]
\end{proposition}
\begin{proof} 
Here is a sketch of the proof. 
Deraux proved in (\cite{Deraux}, Theorem 7.1) 
that the group  $\Delta(\frac{2k+1}{2},\frac{2k+1}{2},\frac{2k+1}{2})$, 
which is generated by the rotations $a,b,c$ around the vertices  of 
the triangle $\Delta$   
embeds into the triangle group associated to a smaller triangle 
$\Delta'$. One constructs $\Delta'$ by considering all 
geodesics of $\Delta$ joining a vertex and the midpoint of its 
opposite side. The three median geodesics pass through the 
barycenter of $\Delta$ and subdivide $\Delta$ into 6 equal triangles.
We can take for $\Delta'$ any one of the 6 triangles of the subdivision. 
It is immediate that $\Delta'$ has angles 
$\frac{\pi}{2k+1}, \frac{\pi}{2}$ and $\frac{\pi}{3}$ so that the associated 
triangle group is $\Delta(2,3,2k+1)$. 
This group has the presentation:  
\[  \Delta(2,3,2k+1)=\langle \alpha, u, v \, ; \, \alpha^{2k+1}=u^3=v^2=\alpha uv=1\rangle, \]
where the generators are the rotations of double angle around the 
vertices of the triangle $\Delta'$.

\begin{lemma}\label{oddlem}
The natural embedding of $\Delta(\frac{2k+1}{2},
\frac{2k+1}{2},\frac{2k+1}{2})$ into 
$\Delta(2,3,2k+1)$ is an isomorphism. 
\end{lemma}
\begin{proof}
A simple geometric computation shows that:  
\[ a= \alpha^2, \, b=v \alpha^2v=u^2\alpha^2u, \, 
c=u\alpha^2 u^2. \]
Therefore $\alpha=a^{k+1}\in \Delta(\frac{2k+1}{2},
\frac{2k+1}{2},\frac{2k+1}{2})$.

\vspace{0.2cm}\noindent
From the relation 
$\alpha u v=1$ we derive $a^{k+1}uv=1$, and 
thus $u=a^kv$. The relation $u^3=1$ reads now 
$a^k(va^kv)a^kv=1$ and replacing $b^k$ by $va^kv$ we find that  
$v=a^kb^ka^k\in \Delta(\frac{2k+1}{2},
\frac{2k+1}{2},\frac{2k+1}{2})$. 

\vspace{0.2cm}\noindent
Further $u=a^kv=a^{-1}b^ka^k\in \Delta(\frac{2k+1}{2},
\frac{2k+1}{2},\frac{2k+1}{2})$. This means that 
$\Delta(\frac{2k+1}{2},
\frac{2k+1}{2},\frac{2k+1}{2})$ is actually  
$\Delta(2,3,2k+1)$, as claimed.
\end{proof}

\vspace{0.2cm}\noindent
It suffices now to find a presentation of $\Delta(2,3,2k+1)$ that uses the 
generators $A=a, B=b$. It is not difficult to 
show that the group  with the presentation of the statement 
is isomorphic to $\Delta(2,3,2k+1)$, the inverse 
homomorphism sending $\alpha$ into $A^{k+1}$, 
$u$ into $A^{-1}B^kA^k$ and $v$ into $A^kB^kA^k$.  
\end{proof}

\vspace{0.2cm}\noindent 
A direct consequence of Propositions \ref{even} and \ref{odd} is the 
following abstract description of the image of Burau's representation: 

\begin{corollary}\label{inftriang}
If  $q$ is not a primitive root of unity of order in   
the set $\{1,2,3,4,6,10\}$,
then $\Gamma_q$ is an infinite triangle group. 
\end{corollary}

\vspace{0.2cm}\noindent 
Alternatively, we obtain a set of normal generators for the kernel 
of Burau's representation, as follows: 

\begin{corollary}\label{kernelburau}
Let $n\not\in\{1, 6\}$ and $q$ a primitive root of unity of order $n$. 
We denote by $N(G)$ the normal closure of a subgroup 
$G$ of $\langle g_1^2,g_2^2\rangle$. 
Then the kernel 
$\ker \beta_{-q}:{\langle g_1^2,g_2^2\rangle\to PGL(2,\C)}$
of the restriction of Burau's representation is given by: 
\[
\left\{\begin{array}{ll} 
N(\langle g_1^{2k},g_2^{2k}, (g_1^2g_2^2)^k\rangle), & {\rm if }\: n=2k;\\
N(\langle g_1^{2(2k+1)},g_2^{2(2k+1)}, (g_1^2g_2^2)^{2k+1},  
(g_1^{-2}g_2^{2k})^{2}, (g_2^{2k}g_1^{2(k-1)})^3\rangle), & {\rm if }\: n=2k+1.\\ 
\end{array}\right.\]  
\end{corollary}

\section{Johnson subgroups and proof of Theorem \ref{Johnson}}

\subsection{Proof of Theorem \ref{Johnson}}\label{pfjo}
For a group $G$ we denote by $G_{(k)}$ the lower central series 
defined by:  
\[ G_{(1)}=G, \, G_{(k+1)}=[G,G_{(k)}], k\geq 1\]
An  interesting  family of subgroups of the mapping class group is 
the set of higher Johnson subgroups defined as follows. 

\begin{definition}
The $k$-th Johnson subgroup 
$I_g(k)$ is the group of mapping classes of homeomorphisms of the closed 
orientable surface $\Sigma_g$ 
whose action by outer automorphisms on $\pi/\pi_{(k+1)}$ is trivial, where 
$\pi=\pi_1(\Sigma_g)$. 
\end{definition}
\vspace{0.2cm}\noindent
Thus $I_g(0)=M_g$, $I_g(1)$ is the Torelli group commonly 
denoted $T_g$, while $I_g(2)$ is the group generated by the Dehn twists along 
separating simple closed curves and considered by Johnson and Morita 
(see e.g. \cite{John,Mor}), which is often denoted by ${K}_g$.

\vspace{0.2cm}
\noindent 
The proof of Theorem \ref{Johnson} follows from the same 
argument as in \cite{F}, where we proved that 
the image of the quantum representation  $\rho_p$ is infinite 
for all $p$ in the given range.  The values of $p$ which are excluded 
correspond to the TQFT's 
${\mathcal V}_3, {\mathcal V}_2, 
{\mathcal V}_{4}, {\mathcal V}_6$, ${\mathcal V}_{8}$ and 
${\mathcal V}_{12}$
and it is known that the images of 
quantum representations are finite 
in these cases.

\vspace{0.2cm}\noindent 
Before we proceed we have to make the cautionary remark that  
$\rho_p$ is only a projective representation.  
Here and henceforth when speaking about Burau's representation we will mean the 
representation $\beta_{q}:B_3\to PGL(2,\C)$ taking values 
in matrices modulo scalars.

\vspace{0.2cm}\noindent 
We will  first consider the generic case where the 
genus is large and the $10$-th roots of unity are discarded.  
This will prove Theorem \ref{Johnson} in most cases. 
Specifically we will prove first:  

\begin{proposition}\label{Johnsonfree}
Assume that $g\geq 4$. Then the 
image $\rho_p(\left(\langle g_1^2, g_2^2\rangle\right)_{(k)})$ contains a 
free non-abelian group for every $k$ and 
$p\not\in\{3, 4, 8, 12, 16, 24, 40\}$. 
\end{proposition}
\begin{proof}
The first step of the proof provides us with enough 
elements of $I_g(k)$ having their support contained in a 
small subsurface of $\Sigma_g$. 

\vspace{0.2cm}\noindent 
Specifically we embed $\Sigma_{0,4}$ into $\Sigma_g$ by means of curves 
$c_1,c_2,c_3,c_4$  as in the figure below. Then the curves $a$ and $b$ 
which  are surrounding two of the holes of $\Sigma_{0,4}$ 
are separating. 

\begin{center}
\includegraphics[scale=0.4]{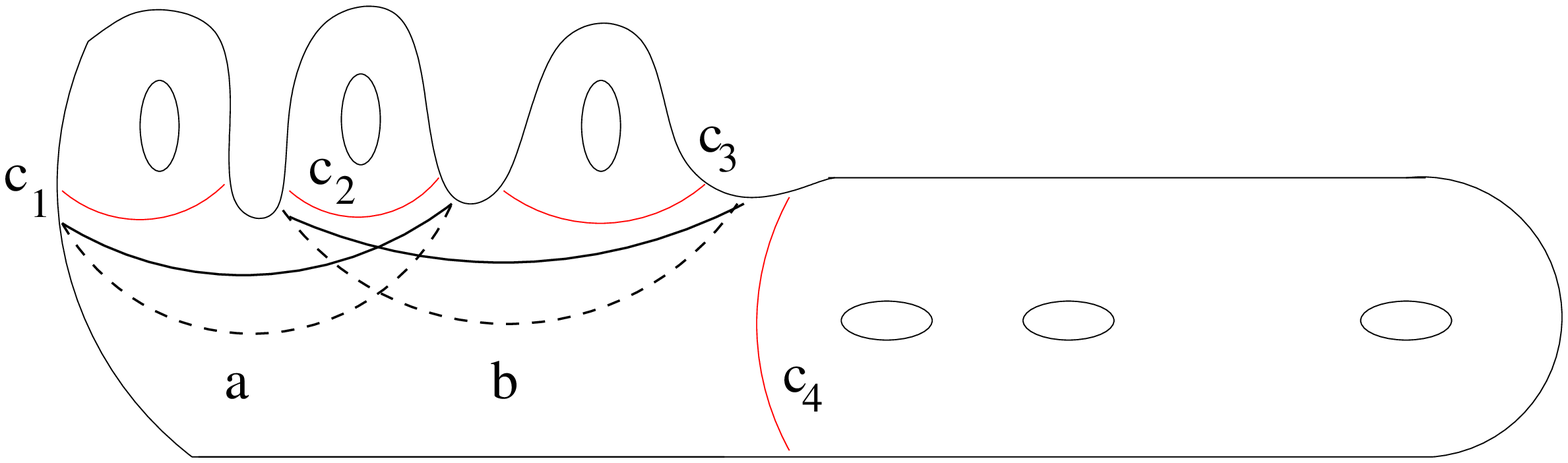}
\end{center}

\vspace{0.2cm}\noindent
The pure braid group $PB_3$ embeds into $M_{0,4}$ using a 
non-canonical splitting of the surjection $M_{0,4}\to PB_3$.  
Furthermore, $M_{0,4}$ embeds into $M_g$ when $g\geq 4$, by using the 
homomorphism induced by the inclusion of 
$\Sigma_{0,4}$ into $\Sigma_g$ as in the figure. 
Then the group generated by the Dehn twists $a$ and $b$ is identified 
with the free subgroup generated by $g_1^2$ and $g_2^2$ into $PB_3$. 
Moreover, $PB_3$ has a natural action on a 
a subspace of the space of conformal blocks
associated to $\Sigma_g$ as in \cite{F}, which is isomorphic 
to the restriction of Burau's representation at some root of unity 
depending on $p$. Notice that the two Dehn twists above are  elements 
of ${K}_g$.  

\vspace{0.2cm}\noindent
We will need the following Proposition whose proof will be given in 
section \ref{pcs}: 

\begin{proposition}\label{cseries} 
The above embedding of $PB_3$ into $M_g$ sends 
$\left(PB_3\right)_{(k)}$ into $I_g(k)$. 
\end{proposition}

\vspace{0.2cm}\noindent 
Recall now that $\langle g_1^2, g_2^2\rangle$ is a normal 
free subgroup of $PB_3$. 
The second ingredient needed in the proof of Proposition 
\ref{Johnsonfree} is the following Proposition which will be proved in 
\ref{jl}:  

\begin{proposition}\label{jolarge}
Assume that $g\geq 4$. Then the 
image $\rho_p(\left(\langle g_1^2, g_2^2\rangle\right)_{(k)})$ contains a 
free non-abelian group for every $k$ and 
$p\not\in\{3, 4, 8, 12, 16, 24, 40\}$. 
\end{proposition}

\vspace{0.2cm}\noindent 
Thus the group $\rho_{p}(\left(PB_3\right)_{(k)})$ contains 
$\rho_{p}(\langle g_1^2,g_2^2\rangle_{(k)})$  and so  
it also contains a free non-abelian group. 
Therefore, Proposition \ref{cseries}  implies
that  $\rho_p(I_g(k))$ contains a free non-abelian subgroup, 
which will complete the proof of  Proposition \ref{Johnsonfree}. 
\end{proof}

\vspace{0.2cm}\noindent 
We further consider the remaining cases and briefly outline  
in section \ref{exc} the modifications 
needed to make the same strategy work  also for small genus surfaces and for those  
values of the parameter $p$  which were excluded above, namely: 
\begin{proposition}\label{excases}
Assume that $g$ and $p$ verify one of the following conditions: 
\begin{enumerate}
\item $g=2$, $p$ is even and $p\not\in\{4,8,12,16, 24,40\}$;
\item $g=3$ and $p\not\in\{3, 4,8,12,16, 24,40\}$;
\item $g\geq 4$ and $p=40$. 
\end{enumerate}
Then $\rho_p(M_g)$ contains a free non-abelian group. 
\end{proposition}

\vspace{0.2cm}\noindent 
Then Propositions \ref{Johnsonfree} and \ref{excases} above will 
prove Theorem \ref{Johnson}.

\subsubsection{Proof of Proposition \ref{cseries}}\label{pcs}
Choose the base point $*$ for the fundamental group $\pi_1(\Sigma_g)$ 
on the circle $c_4$ that separates the sub-surfaces 
$\Sigma_{3,1}$  and $\Sigma_{g-3,1}$. 
Let $\varphi$ be a homeomorphism of $\Sigma_{0,4}$ that is identity 
on the boundary and whose mapping class $b$ belongs to $PB_3\subset M_{0,4}$. 
Consider its extension $\widetilde{\varphi}$ to $\Sigma_g$ 
by identity outside $\Sigma_{0,4}$. Its mapping class $B$ in $M_g$ 
is the image of $b$ in $M_g$.   

\vspace{0.2cm}\noindent 
In order to understand the action of $B$ on $\pi_1(\Sigma_g)$ 
we introduce three kinds of loops based at $*$: 
\begin{enumerate}
\item Loops of type I are those included in $\Sigma_{g-3,1}$. 
\item Loops of type II are those contained in $\Sigma_{0,4}$. 
\item Begin by fixing three simple arcs $\lambda_1,\lambda_2,\lambda_3$ 
embedded in $\Sigma_{0,4}$ joining $*$ to the three other 
boundary components $c_1,c_2$ and $c_3$,  respectively.  
Loops of type III are of the form 
$\lambda_i^{-1}x\lambda_i$, where $x$ is some loop based at the endpoint 
of $\lambda_i$ and contained in the 1-holed torus bounded by $c_i$. 
Thus loops of type III generate $\pi_1(\Sigma_{3,1},*)$.  
\end{enumerate}

\vspace{0.2cm}\noindent 
Now, the action of $B$ on the homotopy classes of loops of type I is trivial. 
The action of $B$ on the homotopy classes of loops of type II is 
completely described by the action of $b\in PB_3$ on $\pi_1(\Sigma_{0,4},*)$. 
Specifically, let $A:B_3\to {\rm Aut}({\mathbb F}_3)$ 
be the Artin representation (see \cite{Birman}). Here ${\mathbb F}_3$ is 
the free group on three generators $x_1,x_2,x_3$ which is identified with 
the fundamental group of the 3-holed disk $\Sigma_{0,4}$.  
\begin{lemma}
If $b\in \left(PB_3\right)_{(k)}$, then $A(b)(x_i)=l_i(b)^{-1}x_i l_i(b)$, where 
$l_i(b)\in \left({\mathbb F}_3\right)_{(k)}$. 
\end{lemma}
\begin{proof}
This is folklore. Moreover, the statement is valid for any number 
$n$ of strands instead of 3. Here is a short proof avoiding heavy computations. 
It is known that the set $PB_{n,k}$ of those pure braids $b$ 
for which the length $m$  Milnor invariants of their  Artin closures 
vanish for all $m\leq k$  is a normal subgroup $PB_{n,k}$ of $B_n$. 
Furthermore, the central series of subgroups $PB_{n,k}$ 
verifies the following (see e.g. \cite{O}): 
\[ [PB_{n,k},PB_{n,m}]\subset PB_{n,k+m}, \, \mbox{\rm for all }\, n,k,m, \] 
and hence, we have $\left(PB_n\right)_{(k)}\subset PB_{n,k}$. 

\vspace{0.2cm}\noindent 
Now,  if $b$ is a pure braid, then  
$A(b)(x_i)=l_i(b)^{-1}x_i l_i(b)$, where $l_i(b)$ is the so-called 
{\em longitude}  of the $i$-th strand. 
Next we can interpret Milnor invariants as coefficients of the 
Magnus expansion of the longitudes. In particular, this correspondence shows 
that $b\in PB_{n,k}$ if and only if $l_i(b)\in \left({\mathbb F}_{n}\right)_{(k)}$. 
This proves the claim.  
\end{proof}

\vspace{0.2cm}\noindent 
The action of $B$ on the homotopy classes of loops of type III can be described 
in a similar way. Let a homotopy class  $a$ of this kind be represented by 
a loop $\lambda_i^{-1}x\lambda_i$. 
Then $\lambda_i^{-1}\varphi(\lambda_i)$ is a loop 
contained in $\Sigma_{0,4}$, whose  homotopy class $\eta_i=\eta_i(b)$ 
depends only on $b$ and 
$\lambda_i$. Then it is easy to see that 
\[ B(a)=\eta_i^{-1} a \eta_i.\]
Let now $y_{i}, z_{i}$ be standard homotopy classes of 
loops based at a point of $c_i$ which generate the fundamental 
group of the holed torus bounded by $c_i$, so that 
$\{y_1,z_1,y_2,z_2,y_3,z_3\}$ is a generator system for 
$\pi_1(\Sigma_{3,1},*)$, 
which is the free group ${\mathbb F}_6$ of rank 6. 
\begin{lemma}
If $b\in \left(PB_3\right)_{(k)}$, then $\eta_i(b)\in 
\left({\mathbb F}_{6}\right)_{(2k)}$. 
\end{lemma}
\begin{proof}
It suffices to observe that 
$\eta_i(b)$ is actually the $i$-th longitude $l_i(b)$ of the braid $b$, 
expressed now in the generators $y_i,z_i$ instead of the generators $x_i$. 
We also know that $x_i=[y_i,z_i]$. Let then 
$\eta:{\mathbb F}_3\to {\mathbb F}_6$ be the group homomorphism 
given on the generators by $\eta(x_i)=[y_i,z_i]$. Then 
$\eta_i(b)=\eta(l_i(b))$. 
Eventually, if $l_i(b)\in \left({\mathbb F}_3\right)_{(k)}$, then 
$\eta(l_i(b))\in \left({\mathbb F}_6\right)_{(2k)}$ and the claim follows. 
\end{proof}

\vspace{0.2cm}
\noindent 
Therefore the class $B$ belongs to $I_g(k)$, since its action on 
every generator of $\pi_1(\Sigma_g,*)$ is a conjugation by an element of 
$\pi_1(\Sigma_g,*)_{(k)}$.

\subsubsection{Proof of Proposition \ref{jolarge}}\label{jl} 
First we want to identify some sub-representation of the 
restriction of $\rho_p$ to $PB_3\subset M_g$. Specifically we have:  

\begin{lemma}\label{contain}
Let $p\geq 5$. The restriction of the quantum representation $\rho_p$ at 
$PB_3\subset M_{0,4}$ has an invariant 2-dimensional 
subspace such that the corresponding  sub-representation 
is equivalent to the Burau representation 
$\beta_{-q_p}$, where the root of unity $q_p$ is given by: 
\[ q_p=\left\{\begin{array}{ll}
-A_p^{-4}=-\exp\left(-\frac{8\pi i}{p}\right), & {\rm if}\: p\equiv 0({\rm mod}\: 4);\\
-A_5^{-4}=-\exp\left(-\frac{4\pi i}{5}\right) , & {\rm if}\: p=5;\\
-A_p^{-8}=-\exp\left(-\frac{8(p+1)\pi i}{p}\right) , & {\rm if}\: p\equiv 1({\rm mod}\: 2), p\geq 7.\\
\end{array}\right. \]
\end{lemma}
\begin{proof}
For even $p$ this is the content of \cite{F}, Prop. 3.2. 
We recall that in this case the invariant 2-dimensional subspace 
is the the space of conformal blocks
associated to the surface $\Sigma_{0,4}$ with all boundary components  
being labeled by the color $1$. The odd case is similar. 
The invariant subspace is the space of conformal blocks 
associated to the surface $\Sigma_{0,4}$ with boundary labels  
$(2,2,2,2)$, when $p=5$ and $(4,2,2,2)$, when $p\geq 7$ 
respectively.    
The eigenvalues of the half-twist  can be computed as in \cite{F}. 
\end{proof}

\vspace{0.2cm}\noindent 
Thus the image $\rho_p(PB_3)$ of the quantum representation 
projects onto the image of the Burau representation 
$\beta_{-q_p}(PB_3)$. 

\vspace{0.2cm}\noindent   
Up to a Galois conjugacy  we can assume that 
$\beta_{-q_p}$ is unitarizable and  after rescaling, it 
takes values in $U(2)$. 
Consider the projection of $\beta_{-q_p}(\left(PB_3\right)_{(k)})$ 
into $U(2)/U(1)=SO(3)$.  

\vspace{0.2cm}\noindent 
A finitely generated subgroup of $SO(3)$ is either 
finite or abelian or else dense in $SO(3)$. If the group is 
dense in $SO(3)$, then it contains a free non-abelian subgroup. 
Moreover, solvable subgroups of $SU(2)$ (and hence of $SO(3)$) are 
abelian. The finite subgroups of $SO(3)$ are well-known. They are 
the following: cyclic groups, dihedral groups, tetrahedral group 
(automorphisms of the  regular tetrahedron), the octahedral group 
(the group of automorphisms of the regular octahedron) 
and the icosahedral group (the group of automorphisms of the 
regular icosahedron or dodecahedron). All but the last one  
are actually solvable groups. The icosahedral group is isomorphic to 
the alternating group $A_5$ and it is well-known that it is simple 
(and thus non-solvable). As a side remark this group appeared in relation 
with the non-solvability of the quintic equation in Felix Klein's monograph
\cite{Klein}.

\begin{lemma}\label{nonsol}
If  $q$ is not a primitive root of unity of order in the set $\{1,2,3,4,6,10\}$,
then $\left(\Gamma_q\right)_{(k)}$ is non-solvable and thus non-abelian 
for any $k$. Moreover, $\left(\Gamma_q\right)_{(k)}$ cannot be $A_5$, 
for any $k$. 
\end{lemma} 
\begin{proof}
If $\left(\Gamma_q\right)_{(k)}$ were solvable,  then 
$\Gamma_q$ would be solvable. 
But one knows that $\Gamma_q$ is not solvable. 
In fact if $q$ is as above, then $\Gamma_{q}$ 
is an infinite triangle group by Corollary \ref{inftriang}. 

\vspace{0.2cm}\noindent 
Now any infinite triangle group has a finite index subgroup which is  
a surface group of genus  at least 2. Therefore, each term of the lower 
central series of that surface group embeds into the corresponding 
term of the lower central series of $\Gamma_q$, so that the 
later is non-trivial.  Since the lower central series of a surface 
group of genus at least 2 consists only of infinite groups it follows 
that no term can be isomorphic to the finite group $A_5$ either.  
\end{proof}

\vspace{0.2cm}\noindent 
Lemma \ref{nonsol} shows that whenever $p$ is 
as in the statement of Proposition \ref{jolarge}, 
the group $\beta_{-q_p}(\left(\langle g_1^2,g_2^2\rangle\right)_{(k)})$ 
is neither finite nor abelian, so that it is dense in $SO(3)$ and hence 
it contains a free non-abelian group. This proves 
Proposition \ref{jolarge}.

\subsubsection{Explicit free subgroups}
The main interest of the elementary arguments in the proof 
presented above is that the 
free non-abelian subgroups in the image are abundant and 
explicit. For instance we have: 

\begin{theorem}\label{free}
Assume that $g\geq 4$, $p\not\in\{3, 4, 12, 16\}$ 
and $p\not\equiv 8({\rm mod}\: 16)$.    
Set $x=\rho_p([g_1^2,g_2^2])$ and $y=\rho_p([g_1^4, g_2^2])$. 
Then the group generated by the iterated commutators 
$[x,[x,[x,\ldots,[x,y]]\ldots]$ and 
$[y,[x,[x,\ldots,[x,y]]\ldots]$ of length $k\geq 3$ 
is a free non-abelian subgroup of 
$\rho_p(I_g(k))$. 
\end{theorem}
\vspace{0.2cm}\noindent 
It is well-known that the order of the matrix  
$\beta_{-q}(g_i)$, $i\in\{1,2\}$ 
in $PGL(2,\C)$ is the order of the root of unity $q$, 
namely the smallest positive $n$ such that $q$ is a primitive root of 
unity of order $n$.
 
\vspace{0.2cm}\noindent 
We considered in Lemma \ref{contain} the root of unity $q_p$ 
with the property that $\beta_{-q_p}$ is a sub-representation 
of the quantum representation $\rho_p$. 
We derive from Lemma \ref{contain} that the order  
of the root of unity $q_p$ is $2o(p)$ where 
\[ o(p)=\left\{ \begin{array}{ll}
\frac{p}{4}, & {\rm if }\: p\equiv 4({\rm mod}\: 8);\\ 
\frac{p}{8}, &{\rm if }\: p\equiv 0({\rm mod}\: 16);\\
\frac{p}{16}, &{\rm if }\: p\equiv 8({\rm mod}\: 16);\\
p,& {\rm if }\: p\equiv 1({\rm mod}\: 2), p\geq 5.\\
\end{array}\right.
\]
Therefore $\beta_{-q_p}(\langle g_1^2,g_2^2\rangle)$ is isomorphic to 
the triangle group $\Delta(o(p),o(p),o(p))$. 
Notice that in general $o(p)\in \frac{1}{2}+\Z$ and 
$o(p)$ is an integer if and only if 
$p\not\equiv 8({\rm mod}\: 16)$,  as we suppose from now on. 
Observe also that the order of $\beta_{-q_p}(g_1^2)$ is a proper 
divisor of the order $p$ of a Dehn twist $\rho_p(g_1^2)$, when $p$ is even. 

\vspace{0.2cm}\noindent 
In the proof of Theorem \ref{free} we 
will need the following result concerning 
the  structure of commutator 
subgroups of triangle groups: 
\begin{lemma}\label{1rel}
The commutator subgroup  $\Delta(r,r,r)_{(2)}$ of a 
triangle group $\Delta(r,r,r)$,  $r\in \Z-\{0,1,2\}$, 
is a 1-relator group  with generators  $\widetilde{c_{ij}}$, for 
$1\leq i,j\leq r-1$, and the relation: 
\[ \widetilde{c_{11}}\cdot \widetilde{c_{21}}^{-1}\cdot \widetilde{c_{22}}\cdot \widetilde{c_{32}}^{-1}\cdots \widetilde{c_{r\,r-1}}^{-1}\cdot \widetilde{c_{rr}}=1.\]
\end{lemma}
\begin{proof}
The  kernel $K$ of the abelianization homomorphism 
$\Z/r\Z*\Z/r\Z\to \Z/r\Z\times\Z/r\Z$  
is the free group generated by the commutators. 
Denote by $\widetilde{a}$ and $\widetilde{b}$ the generators 
of the two copies of the cyclic group $\Z/r\Z$. 
Then $K$ is freely generated by 
$\widetilde{c_{ij}}=[\widetilde{a}^i,\widetilde{b}^j]$, where 
$1\leq i,j\leq r-1$.   
The group $\Delta(r,r,r)$ is the quotient of 
$\Z/r\Z*\Z/r\Z$ by the normal subgroup generated by the element 
$(\widetilde{a}\widetilde{b})^r\,\widetilde{a}^{-r}\widetilde{b}^{-r}$, 
which belongs to $K$. 
This shows that $\Delta(r,r,r)_{(2)}$ is a 
1-relator group, namely the quotient of $K$ by 
the normal subgroup generated by the element 
$(\widetilde{a}\widetilde{b})^r\,\widetilde{a}^{-r}\widetilde{b}^{-r}$. 
In order to get the explicit form of the relation we have to express this 
element as a product of the generators of $K$, i.e., as 
a product of commutators of the form $[\widetilde{a}^i,\widetilde{b}^j]$. 
This can be done as follows:  
\[(\widetilde{a}\widetilde{b})^r\,\widetilde{a}^{-r}\widetilde{b}^{-r}= 
[\widetilde{a},\widetilde{b}][\widetilde{b},\widetilde{a}^2]
[\widetilde{a}^2,\widetilde{b}^2]\cdots 
[\widetilde{a}^{r-1},\widetilde{b}^{r-1}]
[\widetilde{b}^{r-1},\widetilde{a}^{r}] 
[\widetilde{a}^{r},\widetilde{b}^{r}].\]
Therefore $\Delta(r,r,r)_{(2)}$ has a presentation with generators 
$\widetilde{c_{ij}}$, where $1\leq i\leq j\leq r$, and 
the relation in the statement of the lemma.  
\end{proof}

\vspace{0.2cm}\noindent {\em Proof of Theorem \ref{free}.}
Recall now the classical Magnus Freiheitsatz, which 
states that any subgroup of a 
1-relator group which is generated by a 
proper subset of the set of generators involved in the  
cyclically reduced word relator  
is free. 

\vspace{0.2cm}\noindent
Assume now that $o(p)\in\Z$ and $o(p)\geq 4$. 
Then $\beta_{-q_p}([g_1^2,g_2^2])$ and $\beta_{-q_p}([g_1^4,g_2^2])$ are 
the elements $\widetilde{c_{11}}$ and  $\widetilde{c_{21}}$  
of $\Delta(o(p),o(p),o(p))_{(2)}$ respectively. 

\vspace{0.2cm}\noindent
An easy application of the Freiheitsatz to the commutator subgroup 
of the infinite triangle group $\Delta(o(p),o(p),o(p))$ 
gives us that the subgroup generated by 
$\beta_{-q_p}([g_1^2,g_2^2])$ and  $\beta_{-q_p}([g_1^4,g_2^2])$ is free. 
This implies that the subgroup generated by $x$ and $y$ is free.

\vspace{0.2cm}\noindent
Eventually the $k$-th term of the lower central series 
of the group generated by $x$ and $y$ is also a free 
subgroup which is contained into $\rho_p((PB_3)_{(k)})\subset \rho_p(I_g(k))$. 
This proves Theorem \ref{free}. 

\vspace{0.2cm}\noindent
When $p\equiv 8({\rm mod}\: 16)$, $o(p)$ is a half-integer and 
$\beta_{-q_p}(\langle g_1^2,g_2^2\rangle)$ is isomorphic to the triangle 
group $\Delta(2,3, 2o(p))$. If ${\rm gcd}(3,2o(p))=1$, then 
$H_1(\Delta(2,3,2o(p)))=0$, so the central series of this triangle 
group is trivial. Nevertheless the group $\Delta(2,3, 2o(p))$ has 
many normal subgroups of finite index which are surface groups and 
thus contain free subgroups. In particular, any subgroup of infinite index 
of $\Delta(2,3,2o(p))$ is free. 
There is then an extension of the previous result in this case, as 
follows:   

\begin{theorem}\label{free2}
Assume that $g\geq 4$, $p\not\in\{8,24,40\}$ 
and $p\equiv 8({\rm mod}\: 16)$ so that $p=8n$, for odd  
$n=2k+1\geq 7$.    
Consider the following two elements of $\langle g_1^2,g_2^2\rangle$:
\[ s=g_1^{2k}g_2^{2k}g_1^{2(k-k^2)}g_2^{-2}g_1^{2k}g_2^{-2}g_1^{2(k-k^2)}
g_2^{2k}g_1^{2k},\]
 and 
\[ t=
g_1^{2k}g_2^{2k}g_1^{2(k-k^2)}g_2^{-2}g_1^{2(k+1)}g_2^{2}g_1^{2(k+k^2)}
g_2^{2k}g_1^{10k}.
\] 
Let $N(s,t)$ be the normal subgroup generated by $s$ and $t$ 
in  $\langle g_1^2,g_2^2\rangle$. 
Then for any choice of $f(n)$ elements $x_1,x_2,\ldots,x_{f(n)}$
from $N(s,t)$ the  image $\rho_p(\langle x_1,x_2,\ldots,x_{f(n)}\rangle)$ 
is a free group. Here the function $f(n)$ is given by:  
\[ f(n)=|PSL(2,\Z/n\Z)|\cdot \frac{n-6}{6n}\]
and, in particular, when $n$ is prime by: 
\[f(n)=\frac{(n+1)(n-1)(n-6)}{12} \cdot\]
Then the group generated by the iterated commutators 
of length $k\geq 3$ is a free subgroup of 
$\rho_p(I_g(k))$. 
\end{theorem}
\begin{proof}
Observe that the map $PSL(2,\Z)\to PSL(2,\Z/n\Z)$ factors through 
$\Delta(2,3,n)$, namely we have a homomorphism
$\psi: \Delta(2,3,n)\to PSL(2,\Z/n\Z)$ defined by 
\[ \psi(\alpha)=\left(\begin{array}{cc}
1 & -1 \\
0 & 1\\
\end{array}\right), \: \psi(u)=\left(\begin{array}{cc}
1 & -1 \\
1 & 0\\
\end{array}\right), \:\psi(v)=\left(\begin{array}{cc}
0 & -1 \\
1 & 0\\
\end{array}\right).  
\]
The matrices $\psi(\alpha),\psi(u),\psi(v)$ are obviously 
elements of orders $n,3$ and $2$ in $PSL(2,\Z/n\Z)$ respectively. 
It follows that the normal subgroup $K(2,3,n)=\ker \psi$ is 
torsion free, because every torsion element in $\Delta(2,3,n)$ is  
conjugate to some power of  one the generators $\alpha,u,v$ 
(see \cite{HKS}). Therefore $K(2,3,n)$ is a 
surface group, namely the fundamental 
group of a closed orientable surface which 
finitely covers  the fundamental domain of $\Delta(2,3,n)$. 
The Euler characteristic $\chi(K(2,3,n))$ of this Fuchsian group 
can easily be computed by means of the formula: 
\[ \chi(K(2,3,n))=|PSL(2,\Z/n\Z)|\cdot \chi(\Delta(2,3,n)),\]
where the (orbifold) Euler characteristic  $\chi(\Delta(2,3,n))$ 
has the well-known expression:
\[ -\chi(\Delta(2,3,n))=1-
\left(\frac{1}{2}+\frac{1}{3}+\frac{1}{n}\right)=\frac{n-6}{6n}.\] 

\vspace{0.2cm}\noindent 
It is also known that any $-\chi(G)+1$ elements 
of a closed orientable surface group $G$ generate a free 
subgroup of $G$. Thus, in order to establish Theorem \ref{free2}, 
it will suffice to show that  the images of the elements $s,t$ under  
Burau's representation $\beta_{-q_p}$ are normal generators 
of the group $K(2,3,n)$. This is equivalent to show that 
these images correspond to the relations needed to impose in 
$\Delta(2,3,n)$ in order to obtain the quotient $PSL(2,\Z/n\Z)$. 
However, one already  knows presentations for this group (see 
\cite{CG}, Lemma 1 and \cite{Hu}) as follows: 
\[ PSL(2,\Z/n\Z)=\langle \alpha,v,u\: |\: \: \alpha^n=u^3=v^2=1, 
gvgv=g\alpha g^{-1}\alpha^{-4}=1\rangle, \]
for odd $n$, where $g=v\alpha^kv\alpha^{-2}v\alpha^k$.
The first three relations above correspond to the presentation 
of $\Delta(2,3,n)$ and the elements 
$gvgv$ and $g\alpha g\alpha^{-4}$ correspond to the images 
of $s$ and $t$ in $\Delta(2,3,n)$, by using the fact   
that 
$\alpha=a^{k+1}$, $v\alpha^2v=b$, $v=a^kb^ka^k$ (see the proof of Lemma \ref{oddlem}). 
\end{proof}

\subsubsection{Second proof of Proposition \ref{Johnsonfree}}\label{shorter}  
We outline here an alternative proof which does not rely 
on the description of the image of Burau's representation 
in Corollary \ref{inftriang}.  This proof is shorter 
but less effective since it does not produce explicit free subgroups
and uses the result of \cite{LW} and the Tits alternative, which need 
more sophisticated tools from the theory of algebraic groups. 

\vspace{0.2cm}\noindent
The image $\rho_p(M_g)$ in $PU(N(p,g))$ is  dense in 
$PSU(N(p,g))$, if $p\geq 5$ is prime (see \cite{LW}), where 
$N(p,g)$ denotes the dimension of the  space of conformal blocks in genus 
$g$ for the TQFT ${\mathcal V}_p$. 
In particular, the image  of $\rho_p$ is Zariski dense in $PU(N(p,g))$. 
By the Tits alternative (see \cite{Tits}) the image is either solvable or 
else it contains a free non-abelian subgroup. However, if the image were 
solvable, then its Zariski closure would be a solvable Lie group, 
which is a contradiction. 
This implies that  $\rho_p(M_g)$ contains a 
free non-abelian subgroup. 

\vspace{0.2cm}\noindent
If $p$  is not prime but has a prime factor $r\geq 5$, then the claim 
for $p$ follows from that for $r$.  
If $p$ does not satisfy this condition, then we have again to use the 
result of Proposition \ref{jolarge} for $k=1$. 
This result can be obtained directly 
from the computations in \cite{Jones} proving that the 
image of the Jones representation of $B_3$ is neither finite nor abelian 
for the considered values of $p$. 
This settles the case $k=1$ of Proposition \ref{Johnsonfree}.  

\vspace{0.2cm}\noindent 
Furthermore, the group $\rho_p(T_g)$ is of finite index in $\rho_p(M_g)$, 
and hence it also contains a 
free non-abelian subgroup. Results of Morita (see \cite{Mor2}) 
show that for $k\geq 2$ 
the group $I_g(k+1)$ is the kernel of the $k$-th 
Johnson homomorphism  $I_g(k)\to A_k$, where 
$A_k$ is a finitely generated abelian group. 
This implies that $[I_g(k), I_g(k)]\subset I_g(k+1)$, for every $k\geq 2$.   
In particular, the $k$-th term of the derived series of 
$\rho_p(T_g)$ is contained into $\rho_p(I_g(k+1))$. 
But every term of the derived series of $\rho_p(T_g)$  contains the 
corresponding term of the derived series of a free subgroup 
and hence a free non-abelian group. 
This proves Proposition \ref{Johnsonfree}.

\begin{remark}
Using the strong version of Tits' theorem due to 
Breuillard and Gelander (see \cite{BG}) 
there exists a free non-abelian subgroup  of $M_g/M_g[p]$ whose 
image in  $PSU(N(p,g))$ is dense. Here $M_g[p]$ denotes the 
(normal) subgroup generated by the $p$-th powers of Dehn twists.  
\end{remark}

\subsection{Proof of Proposition \ref{excases}}\label{exc}
If the genus $g\in \{2,3\}$, then the construction used in the proof of 
Proposition \ref{Johnsonfree} should be modified. This is equally valid 
when we want to get rid of the values $p=5$ and $p=40$. 

\vspace{0.2cm}\noindent
The proof follows along the same lines as Proposition \ref{jolarge}, 
but the embeddings $\Sigma_{0,4}\subset \Sigma_g$ are now different. 
In all cases  considered below the analogue of Proposition 
\ref{cseries} will still be true, namely the image of the subgroup 
$\langle g_1^2,g_2^2\rangle_{(k)}$ by the homomorphisms 
$M_{0,4}\to M_g$ will be contained within the Johnson subgroup $I_g(k)$. 

\vspace{0.2cm}\noindent
If $g=2$ we use the following embedding 
$\Sigma_{0,4}\subset \Sigma_2$:

\begin{center}
\includegraphics[scale=0.4]{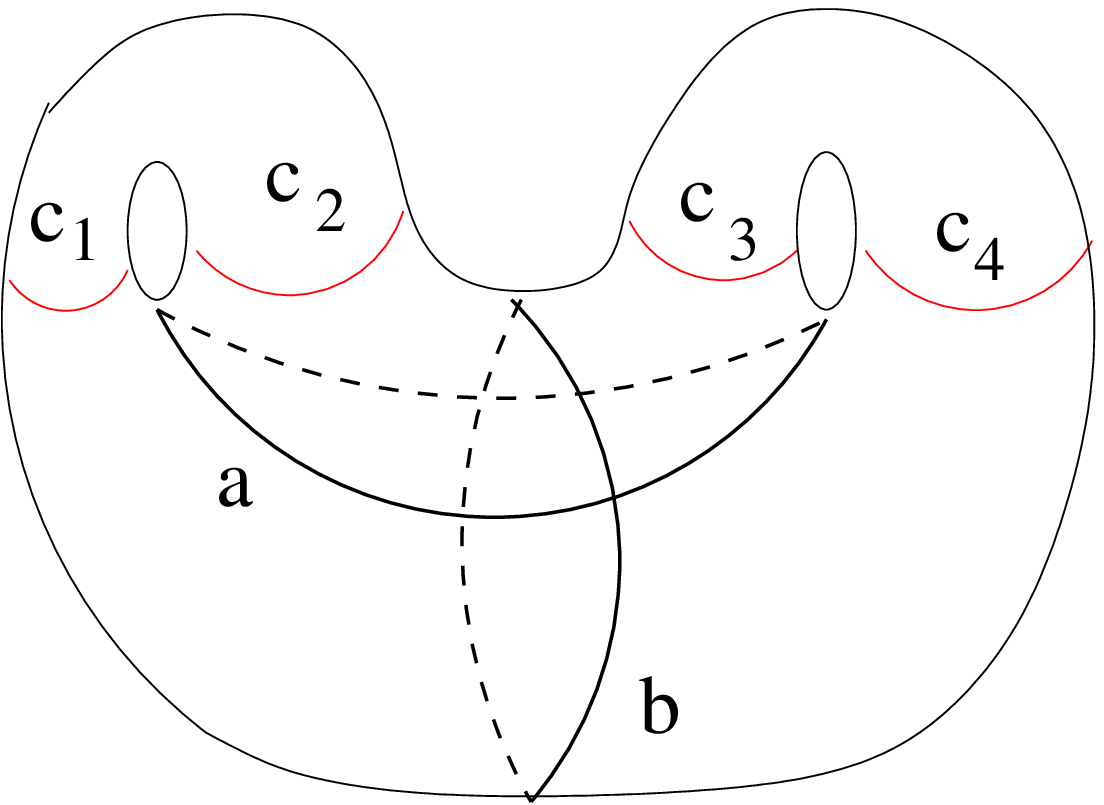}
\end{center}

\vspace{0.2cm}\noindent
Although the homomorphism $M_{0,4}\to M_2$  induced by this 
embedding is not anymore injective, it sends  
the free subgroup $\langle g_1^2,g_2^2\rangle\subset PB_3\subset M_{0,4}$ 
isomorphically 
onto the subgroup of $M_2$ generated by the Dehn twists 
along the curves $a$ and $b$  in the figure above. 

\vspace{0.2cm}\noindent
Consider  for even $p$ the space of conformal blocks 
associated to $\Sigma_{0,4}$ with boundary 
labels $(1,1,1,1)$. This 2-dimensional subspace is 
$\rho_p(\langle g_1^2,g_2^2\rangle)$-invariant 
and the restriction of $\rho_p$ to this subspace is still equivalent 
to Burau's representation $\beta_{-q_p}$ (see \cite{F}). 
Therefore Proposition \ref{jolarge} shows that 
$\rho_p(\langle g_1^2,g_2^2\rangle_{(k)})$,  and 
hence also $\rho_p(I_2(k))$, contains a free non-abelian group.

\vspace{0.2cm}\noindent
If $g=3$ and $p\geq 7$ is odd, then 
we consider the following embedding 
of $\Sigma_{0,4}\subset M_3$: 

\begin{center}
\includegraphics[scale=0.4]{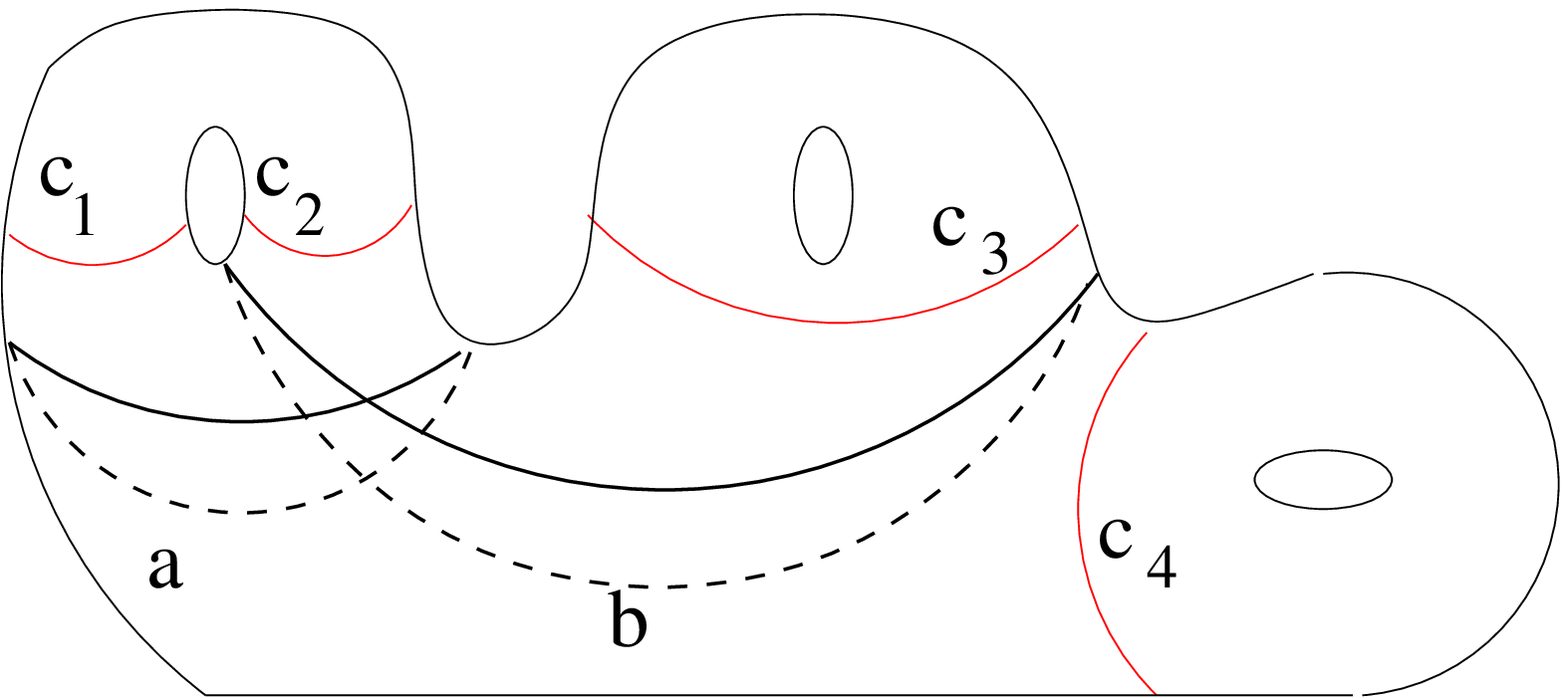}
\end{center}  

\vspace{0.2cm}\noindent
The homomorphism $M_{0,4}\to M_3$  induced by this 
embedding is not injective but it also sends  
the free subgroup $\langle g_1^2,g_2^2\rangle\subset PB_3\subset M_{0,4}$ 
isomorphically 
onto the subgroup of $M_3$ generated by the Dehn twists 
along the curves $a$ and $b$  in the figure above. 
The space of conformal blocks 
associated to $\Sigma_{0,4}$ with boundary 
labels $(2,2,2,4)$ is a 2-dimensional subspace invariant by 
$\rho_p(\langle g_1^2,g_2^2\rangle)$ 
and the restriction of $\rho_p$ to this subspace is equivalent 
to Burau's representation $\beta_{-q_p}$. 
Applying again Proposition \ref{jolarge}, we find that 
$\rho_p(\langle g_1^2,g_2^2\rangle_{(k)})$,  and 
hence $\rho_p(I_3(k))$, contains a free non-abelian group. 
This also gives the desired results for any $g\geq 3$, and 
$p$ as in the statement. 

\vspace{0.2cm}\noindent
Eventually we have to settle the case $p=40$, when 
$\beta_{-q_p}(B_3)$ is known to have finite image 
(see \cite{Jones}). We will consider instead the 
representation $\rho_p(i(PB_4))$, where 
$PB_4$ embeds non-canonically into $M_{0,5}$ and 
$M_{0,5}$ maps into $M_3$ 
by the homomorphism $i:M_{0,5}\to M_3$ induced by 
the inclusion  $\Sigma_{0,5}\subset \Sigma_g$ drawn below: 

\begin{center}
\includegraphics[scale=0.4]{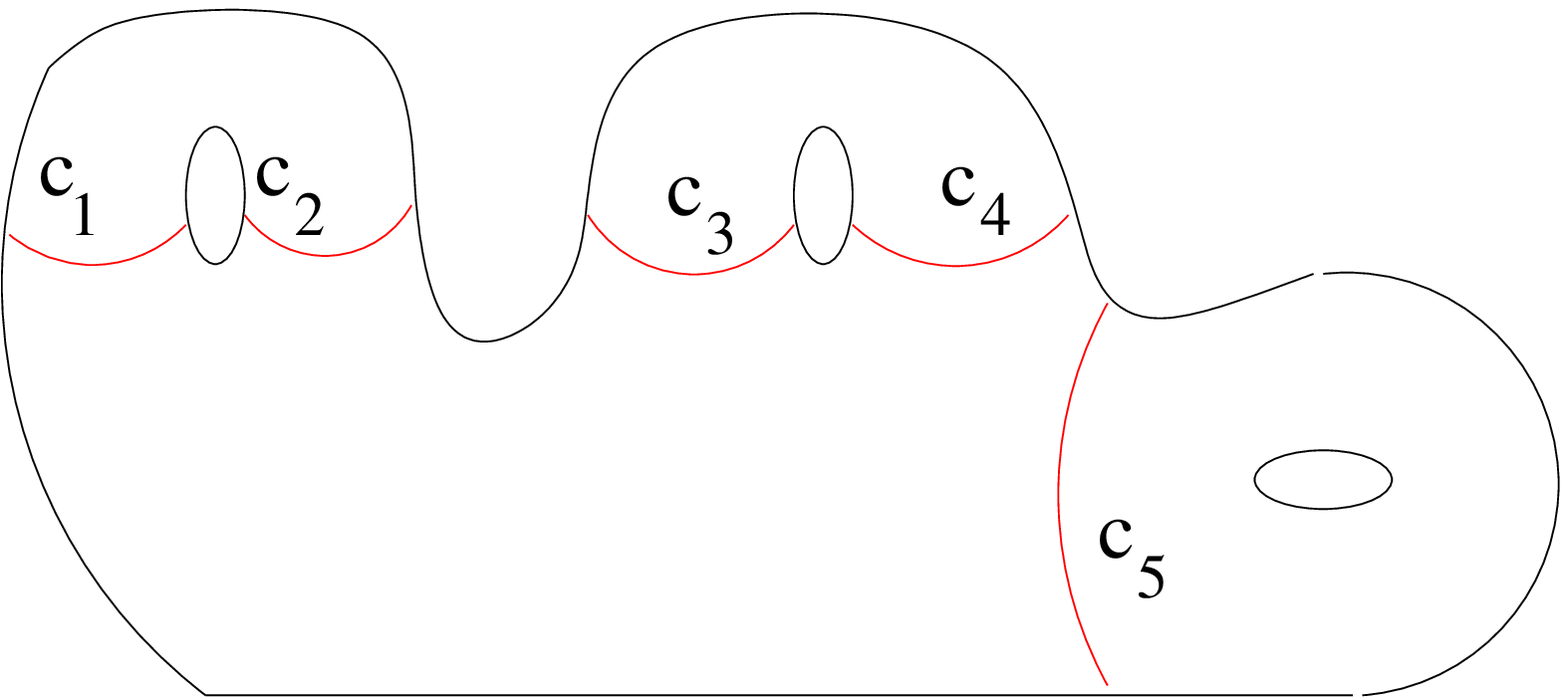}
\end{center}  

\vspace{0.2cm}\noindent
We consider the 3-dimensional space of conformal blocks associated 
to the surface $\Sigma_{0,5}$ with the boundary labels 
$(1,1,1,1,2)$, when $p=40$ and the labels 
$(2,2,2,2,2)$, when $p=5$ respectively. This space of conformal blocks is  
$\rho_p(i(PB_4))$-invariant. The restriction of 
${\rho_p}|_{PB_4}$ to this invariant subspace is known 
(see again \cite{F}) to be equivalent to the Jones 
representation of $B_4$ at the corresponding root of unity.   

\vspace{0.2cm}\noindent
Now we have to use a result of Freedman, Larsen and Wang 
(see\cite{FLW}) subsequently reproved and extended by Kuperberg in 
(\cite{Ku}, Thm.1) saying that the Jones representation of $B_4$  
at a $10$-th root of unity on 
the two 3-dimensional conformal blocks we chose  
is Zariski dense in the group $SL(3,\C)$. 
A particular case of the  Tits alternative 
says that any finitely generated subgroup 
of $SL(3,\C)$  is either solvable or else contains a free non-abelian group. 
A solvable subgroup  has also a solvable Zariski closure. 
The denseness result from above implies then 
that ${\rho_p}(PB_4)$, and hence also $\rho_p((PB_4)_{(k)})$, 
contains a free non-abelian group. The arguments in the proof of 
Proposition \ref{jolarge} carry on to this setting and this 
proves Proposition  \ref{excases}.

\begin{corollary}
For any $k$  the quotient group $I_g(k)/M_g[p]\cap I_g(k)$, 
and in particular,  ${K}_g/{K}_g[p]$, for $g\geq 3$ 
and $p\not\in \{3, 4, 8, 12, 16, 24\}$   
contains a free non-abelian subgroup. Here ${K}_g[p]$ is the 
normal subgroup of $K_g$ generated by the $p$-th powers of the Dehn twist along 
separating simple closed curves.  
\end{corollary}

{
\small      
      
\bibliographystyle{plain}

}

\end{document}